\documentclass{amsart}
\usepackage[utf8]{inputenc}
\usepackage{amssymb}
\usepackage{amsmath}
\usepackage{graphics}
\usepackage{dsfont}

\usepackage[bookmarksnumbered,colorlinks]{hyperref}

\newtheorem{theorem}{Theorem}[section]
\newtheorem{proposition}[theorem]{Proposition}
\newtheorem{corollary}[theorem]{Corollary}
\newtheorem{lemma}[theorem]{Lemma}

\newtheorem{definition}[theorem]{Definition}

\newtheorem{remark}[theorem]{Remark}

\numberwithin{equation}{section}

\newcommand{\R}{\mathds R}

\def\br#1\er{\textcolor{red}{#1}} %
      \def\bb#1\eb{\textcolor{blue}{#1}} %

\begin{document}

\title[A note on the existence of tubular neighbourhoods on Finsler manifolds]{A note on the existence of tubular neighbourhoods on Finsler manifolds and minimization of orthogonal geodesics to a submanifold}


\author[B. Alves]{Benigno Alves}

\author[M. A. Javaloyes]{Miguel Angel Javaloyes }



\address{Benigno O. Alves \hfill\break\indent 
Instituto de Matem\'{a}tica e Estat\'{\i}stica\\
Universidade de S\~{a}o Paulo, \hfill\break\indent
 Rua do Mat\~{a}o 1010,05508 090 S\~{a}o Paulo, Brazil}
\email{gguialves@hotmail.com,  benigno@ime.usp.br}

\address{Miguel Angel Javaloyes \hfill\break\indent 
Departamento de Matem\'aticas, \hfill\break\indent
Universidad de Murcia, \hfill\break\indent
Campus de Espinardo,\hfill\break\indent
30100 Espinardo, Murcia, Spain}
\email{majava@um.es}

\thanks{The first author was supported by CNPq (PhD fellowship) and partially
supported by  PDSE-Capes (PhD sandwich program). The second author was partially supported  by Spanish  MINECO/FEDER project reference
MTM2015-65430-P and Fundaci\'on S\'eneca
(Regi\'on de Murcia) project 19901/GERM/15  }






\begin{abstract}
In this note, we prove that given a submanifold $P$  in a Finsler manifold $(M,F)$, (i) the orthogonal geodesics to $P$ minimize the distance from $P$ at least in some interval, (ii) there exist tubular neighbourhoods around each point of $P$, (iii) the distance from $P$ is smooth in some open neighbourhood of $P$ (but not necessarily in $P$).
\end{abstract}

\maketitle

\section{Introduction}
One of the main properties of geodesics of  Riemannian (or more generally Finsler) manifolds is that they locally minimize the distance associated with the metric (see \eqref{dist}). To be more precise, if $(M,F)$ is a Finsler manifold and $\gamma:[a,b]\rightarrow M$ is a geodesic of $(M,F)$, then there exists $\varepsilon\in \R$, with $b-a\geq \varepsilon>0$, such that $\gamma |_{[a,a+\varepsilon]}$ minimizes the distance from $\gamma(a)$ to  $\gamma(a+\varepsilon)$. It is well-known that given a submanifold $P\subset M$ in a Riemannian manifold $(M,g)$, orthogonal geodesics to $P$ also minimize the distance. In order to prove this, the standard procedure is to consider a tubular neighbourhood $U$ of $P$ and to proceed as follows: let  $\bar B(x,r)$ be the closed ball of center $x\in M$ and radius $r>0$ for the distance associated with the Riemannian metric. Then given a geodesic $\gamma:[a,b]\rightarrow M$ such that $\gamma(a)\in P$ and $\dot\gamma(a)$ is orthogonal to $P$, if we choose an $\varepsilon>0$ such that $\gamma([a,a+\varepsilon])\subset U$ and $\bar B(\gamma(a+\varepsilon),\varepsilon)\subset U$,  it follows straightforward that $\gamma|_{[a,a+\varepsilon]}$ minimizes the distance from $P$ to $\gamma(a+\varepsilon)$. As  tubular neighbourhoods always exist in Riemannian Geometry and it is always possible to find an $\varepsilon>0$ with that property, this shows that orthogonal geodesics to submanifolds are always locally minimizing. The main difficulty to generalize this result to Finsler manifolds is that the existence of tubular neighbourhoods does not follow from the classical results for vector bundles as in \cite[IV,\S5]{Lang}. This is because the subset of orthogonal vectors to a submanifold is not in general a vector bundle. In order to overcome this problem we will use the following strategy:
\subsection{Sketch of the proof of Theorem \ref{TeoremaPrincipal}} Observe first that when we consider an (orientable) hypersuface $\tilde P$ in a Finsler manifold $(M,F)$, the orthogonal  vectors determine two vector bundles (one on each side of the hypersurface) and we can apply the classical result about the existence of tubular neighbourhoods. Given an arbitrary submanifold $P$ and a geodesic $\gamma$ orthogonal to $P$, we will find a hypersuface $\tilde P$ that contains $P$ and it is also orthogonal to $\gamma$. Then using the tubular neighbourhood, we can prove that $\gamma$ minimizes the distance to $\tilde P$ in some interval. As $P$ is contained in $\tilde P$, $\gamma$ also minimizes the distance from $P$ in the same interval. In order to obtain the other two properties announced in the abstract, namely, the existence of tubular neighbourhoods and the smoothness of the distance from the submanifold, we will need to accomplish this process with continuous dependence on (the initial vector of) the geodesic and the length of the  minimizing interval. 

\subsection{Existence of geometric tubular neighbourhoods (Theorem \ref{tubneigh})}  Once we have proved that orthogonal geodesics minimize in some interval, and that this interval can be fixed in geodesics with close initial unit velocity (Theorem \ref{TeoremaPrincipal}), it follows (using Lemma \ref{lemma3}) that the exponential is injective. In order to show that is a local diffeomorphism, we need additionally to ensure that there are no focal points for some common interval to all unit orthogonal geodesics (Proposition \ref{focalpoints}). This result together with the injectivity of the exponential map allows us to obtain the existence of geometric tubular neighbourhoods.  

Finally, let us notice the importance of tubular neighbourhoods in Finsler manifolds. They have been considered in \cite[\S 16.2]{Shen01} and \cite{Wu} to compute the volume under some additional hypotheses and they are also fundamental in the study of singular Finsler foliations \cite{AAJ17}. 

\section{Preliminaries}
Our notion of Finsler metric coincides with the one given in the classical book \cite{BaChSh00}, namely, given a manifold $M$ of dimension $n$, a Finsler metric is a non-negative function on the tangent bundle $F:TM\rightarrow [0,+\infty)$ such that $F$ is smooth in $TM\setminus \bf 0$, positive homogeneous of degree 1, that is, $F(\lambda v)=\lambda F(v)$ for every $\lambda >0$ and $v\in TM$, and for every $v\in TM\setminus\bf 0$, the fundamental tensor of $F^2$, defined as 
\begin{equation}\label{fundtensor}
 g_{v}(u,w)=\frac{1}{2}\frac{\partial^{2}}{\partial t\partial  s} F^{2}(v+tu+sw)|_{t=s=0}
 \end{equation}
for any $u,w\in T_{\pi(v)}M$,  is  positive-definite, where $\pi:TM\rightarrow M$ is the natural projection. With the Finsler metric at hand, we can define the length of any (piecewise) smooth curve  $\gamma:[a,b]\rightarrow M$ as $\ell_F(\gamma)=\int^b_a F(\dot\gamma) ds$ and the distance between any pair of points $p,q\in M$ as
\begin{equation}\label{dist}
 d_F(p,q)=\inf_{\gamma\in C_{p,q}} \ell_F(\gamma),
 \end{equation}
where $C_{p,q}$ is the space of (piecewise) smooth curves $\gamma:[a,b]\rightarrow M$ from $p$ to $q$. Observe that this distance is not necessarily symmetric in general, namely, $d_F(q,p)\not=d_F(p,q)$  as it is allowed $F(-v)\not=F(v)$ where $v\in TM$. As a consequence for every $p\in M$ and $r>0$, one has to handle to types of balls, {\it forward balls}, $B^+_F(p,r)=\{q\in M: d_F(p,q)<r\}$, and {\it backward balls}, $B^-_F(p,r)=\{q\in M: d_F(q,p)<r\}$. As the distance associated with a Finsler metric is continuous, the respective closures of the forward and backward balls are given by $\bar B^+_F(p,r)=\{q\in M: d_F(p,q)\leq r\}$ and $\bar B^-_F(p,r)=\{q\in M: d_F(q,p)\leq r\}$.
It is well-known that if a curve $\gamma:[a,b]\rightarrow M$ is minimizing, namely, $\ell_F(\gamma)=d_F(\gamma(a),\gamma(b))$, then $\gamma$ is a geodesic. Recall that geodesics can be defined as the critical points of the energy function 
\[E_F(\gamma)=\int^b_a F(\dot\gamma)^2 ds\]
when restricted to $C_{p,q}$. Here we consider geodesics with affine parameter, namely, they can be affinely parametrized to an arclength curve. They can also be defined with some of the connections at hand in Finsler Geometry: Chern, Cartan, Bewald, Hashiguchi... (see \cite[page 39]{BaChSh00}). We will use here the Chern connection viewed as a family of affine connections \cite{Jav14}. It turns out that for every $p\in M$ and $v\in T_pM$, there exists a unique (smooth) geodesic $\gamma:[0,b)\rightarrow M$ such that $\gamma(0)=p$ and $\dot\gamma(0)=v$, being inextendible to the right.
On the other hand, not every geodesic $\gamma:[a,b]\rightarrow M$ is minimizing, but it is locally minimizing, which means that there exists $\varepsilon$, with $0<\varepsilon\leq b-a$, such that $\gamma|_{[a,a+\varepsilon]}$ is minimizing. Given any subset $S$ and a point $q\in M\setminus S$, we can define the distance $d_F(S,q)=\inf_{p\in S}d_F(p,q)$. We will focus here in the case in that $S$ is a submanifold of $M$ which we will denote by $P$. Observe that if a curve $\gamma:[a,b]\rightarrow M$ is minimizing from $P$ to $q$, namely, $\ell_F(\gamma)=d_F(P,q)$, then it is a geodesic and $g_{\dot\gamma(a)}(\dot\gamma(a),w)=0$ for every $w\in T_{\gamma(a)}P$ (see for example \cite[Remark 2.1 (ii)]{CJM11}). This is a motivation for the concept of orthogonal vectors to a submanifold $P$. We will say that a vector $v\in TM\setminus \bf 0$ is {\it orthogonal to $P$} if $\pi(v)\in P$ and $g_v(v,w)=0$ for every $w\in T_{\pi(v)}P$. We will denote by $\nu(P)$ the subset of orthogonal vectors to $P$, which is a cone in the sense that if $v\in \nu(P)$, then $\lambda v\in \nu(P)$ for every $\lambda>0$. Observe that $\nu(P)$ is not a vector bundle over $P$ as in classical Riemannian Geometry, but $\nu(P)$ is a smooth submanifold of $TM$ of dimension $n=\dim M$,  and the restriction $\pi:\nu(P)\rightarrow P$ is a submersion \cite[Lemma 3.3]{JavSoa15}. Observe also that even though the {\it orthogonal cone} $\nu(P)$ is not a vector bundle over $P$, it is related to a vector bundle over $P$ by means of the Legendre transform $\mathcal L: TM\setminus {\bf 0}\rightarrow TM^*\setminus {\bf 0}$, which is defined as ${\mathcal L}(v)=g_v(v,\cdot)$ for every $v\in TM\setminus \bf 0$. Indeed, $\nu(P)$ is the inverse image by $\mathcal L$ of the annihilator of the tangent bundle of $P$ (as a subset of $TM$). For a related result see \cite[Prop. 2.4]{AlvDur01}. 
Furthermore, we will denote by $\nu^1(P)$ the subset of unit orthogonal vectors to $P$,  which is a smooth submanifold of dimension $n-1$ as it is the intersection of two transversal submanifolds: the unit bundle and $\nu(P)$. 

A further concept which will be essential in our results is the exponential map associated with a Finsler metric. The {\it exponential map} $\exp:\mathcal U\subset TM\rightarrow M$ is defined as follows. Given $v\in TM\setminus \bf 0$, let us denote by $\gamma_v:[0,b)\rightarrow M$ the  geodesic (inextendible to the right) with $\gamma_v(0)=\pi(v)$ and $\dot\gamma_v(0)=v$. Then  $v\in \mathcal U$ (the domain of the exponential map)  if $b>1$ and in this case $\exp(v)=\gamma_v(1)$. Moreover, by the smooth dependence of geodesics on the Cauchy initial conditions, one can prove that is $C^1$ in its domain and $C^\infty$ away from the zero section (see for example \cite[\S 5.3]{BaChSh00}).

We are now ready to give one of the most fundamental notions of this note.
\begin{definition}
We say that a neighbourhood $V$ of $P$ is a \emph{tubular neighbourhood} of $P$ if there exists an open neighbourhood $\mathcal V\subset \nu(P)$ such that $\exp:{\mathcal V}\rightarrow V\setminus P$ is a diffeomorphism.
\end{definition}
To guarantee the existence of tubular neighbourhoods allows us to prove the local minimization of orthogonal geodesics. Nevertheless, we will prove first the local minimization of orthogonal geodesics due to problems related with the non-smoothness of (the extension of) $\nu(P)$ in the zero section, namely, as said before, it is a cone but not a vector bundle. In some cases, a more specific notion of tubular neighbourhood is needed in such a way that all the orthogonal geodesics are minimizing. Let us introduce the notion of {\it cut value of $P$}, which will play a relevant role in the following.
Given $v\in \nu^1(P)$, the cut value of $v$ is defined by
$$ i(v)=i_P(v):=\sup\{t; d_F(P,\gamma_v(t))=\ell_F(\gamma_v|_{[0,t]})=t \}. $$
If $i(v)<+\infty$, the point $\gamma_v(i(v))$ is said to be a \emph{cut point} of $P$. And the set of all cut points of $P$ is called the \emph{cut locus} of $P$ and it will be denoted by $Cut(P).$

\begin{lemma}\label{lemma3} If $v,w \in \nu^1(L)$ and  $t_0,t_1< \min\{ i(v),i(w)\}$ such that $\gamma_v(t_0)=\gamma_w(t_1)$, then $v=w$.
\end{lemma}
\begin{proof}
First observe that $t_0=t_1$, otherwise one of the geodesics $\gamma_v$ and $\gamma_w$ is not minimizing in $t_0$ or $t_1$, but this is a contradiction with the hypothesis  $t_0,t_1< \min\{ i(v),i(w)\}$. Take  $t_2\in (t_0,i(v))$. For definition  $d_F(P,\gamma_v(t_2))=t_2 $. Let $\alpha$ be the concatenation of $\gamma_w|_{[0,t_0]} $ with $\gamma_v|_{[t_0,t_2]}$. The length of $\alpha$ is $\ell_F(\alpha)=t_2.$
In particular, the length of $\alpha$ is equal to the distance from $\alpha(0)$ to $\alpha(t_2)$. Then $\alpha$ is a geodesic and there is no break. This implies that $\gamma_v=\gamma_w$ and in particular $v=w$.
\end{proof}





\section{The Existence}

\begin{theorem}\label{TeoremaPrincipal}
Let $P$ be a k-dimensional embedded submanifold in a Finsler manifold $(M^n,F)$. Given $v_0\in \nu^1(P)$, there exists a neighbourhood of $v_0$ in ${\mathcal U}\subset\nu^1(P)$ and $\varepsilon>0$ such that $i(v)>\varepsilon$ for every $v\in \mathcal U$.

\end{theorem}

\begin{proof}

%

First, we will show that for every $v\in \nu^1(P)$ close to $v_0$ there exists a hypersurface $\tilde P_v$ such that $P\subset\tilde P_v$ and $v$ is orthogonal to $\tilde P_v$. Additionally, we will see that the unit orthogonal vector fields that extend $v$ have a smooth dependence (see the map $N_+$ in \eqref{vectorN}). Being $\tilde P_v$ a hypersurface, the orthogonal vectors determine two vector bundles (one in each orientation of $\tilde P_v$) allowing to follow the classical approach. This will lead us to prove that the exponential map restricted to   orthogonal vectors to $\tilde P_v$ (see \eqref{expv}) is a diffeomorphism in a neighbourhood of the zero section. Finally, with the help of a small enough convex neighbourhood, we will prove that $\gamma_v$ minimizes in  a certain interval $[0,\varepsilon]$ for all $v$ in a neighbourhood of $v_0$.

{\it Step 1: to find the hypersurfaces $\tilde P_v$ and to prove the smoothness of their orthogonal unit vectors.}  Let $\varphi:\R^{n}\to W\subset M$ be a parametrization of $M$ around $p_0=\pi(v_0)$ such that $\varphi(\R ^k\times 0)=P\cap W$. Consider a basis of $\R^n$ of the form $\{d\varphi^{-1}(v_0), e_2,\ldots, e_n\}$ and some open neighbourhood  $\tilde{\mathcal U}\subset \nu^1(P)$ of $v_0$ such that $B_v=\{d\varphi^{-1}(v), e_2,\ldots, e_n\}$ is a basis for every $v\in \tilde{\mathcal U}$. With the  positive definite scalar product $h_v=\varphi^*(g_v)$, apply Gram-Schmidt to the basis $B_v$ in order to obtain a basis  in $\R^n$, $\{E_1(v)=d\varphi^{-1}(v),E_2(v),\ldots,E_n(v)\}$, in such a way that the last $n-1$ vectors  are $h_v$-orthogonal to $d\varphi^{-1}(v)$ and they depend smoothly on $v$. Now define the map $\psi: \R^{n}\times \tilde{\mathcal U}\rightarrow  \R^n$, given by $\psi(x,v)=\sum_{i=1}^{ n}x_iE_{ i}(v)$ where $x=(x_1,\ldots,x_{n})\in \R^{n}$ and $v\in\tilde{\mathcal U}$. Observe that $\tilde P_v=\varphi\circ\psi( \{0\}\times  \R^{n-1}\times v)$ is a smooth hypersurface of $M$ which is orthogonal to $v$. We want to prove that the orthogonal unit vectors to the hypersurfaces $\tilde P_v$ (which extend $v$) depend smoothly on $v$.  We will use that the Legendre transformation maps the orthogonal vectors to $\tilde P_v$ to the annihilator of $T\tilde P_v$, which is a vector bundle.  Define $Y_{\pm}:\R^{n-1}\times\tilde{\mathcal U}\rightarrow TM$  as $Y_{\pm}(x,v)={\mathcal L}^{-1}(\phi^v_*({dx_1}_{(0,x)}))$,
where, for every $v\in \tilde{\mathcal U}$, $\phi^v:\R^n\rightarrow W$ is defined as $\phi^v(x)=\varphi(\psi(x,v))$,  ${dx_1}_{(0,x)}$ is the covector associated with the first coordinate of $\R^n$ in the tangent space to $(0,x)\in \R^n$,  and ${\mathcal L}:TM\setminus 0\rightarrow TM^*\setminus 0$ is the Legendre transformation associated with $F$. 
  Finally, we get two smooth maps
$N_\pm:\R^{n-1}\times\tilde{\mathcal U}\rightarrow TM$  given by 
\begin{equation}\label{vectorN}
 N_\pm(x,v)=\frac{1}{F(Y_\pm(x,v))} Y_\pm(x,v),
 \end{equation}
in such a way that for every $v\in \tilde{\mathcal U}$,  $N_\pm(x,v)\in \nu^1(\tilde P_v)$ and $v=N_+(x(v),v)$ for some $x(v)\in \R^{n-1}$.

{\it Step 2: to prove that the exponential is a diffeomorphism in an open subset of the normal bundle to $\tilde P_v$.} Define  the map $E^\pm: \tilde{W}_\pm\subset \R\times \mathbb{R}^{n-1}\times \tilde{\mathcal U} \to \tilde{\mathcal U}\times M$ as
\begin{equation}\label{expv}
 E^\pm(s,x,v)= (v,\gamma_{N_\pm(x,v)}(s)),
 \end{equation} 
where $\gamma_{N_\pm(x,v)}$ is the geodesic with initial velocity in $s=0$ equal to $N_\pm(x,v)$ and $\tilde{W}_\pm$ is an open subset such that $\gamma_{N_\pm(x,v)}(s)$ is well-defined for each $(s,x,v)\in \tilde{W}_\pm$.  Observe that  $dE^\pm_{(0,x_0,v_0)}$ is a linear isomorphism, where $x_0$ is determined by $\varphi(\psi(x_0,v_0))=p_0$.  Then  by the Inverse function Theorem, there exists an open subset $ {\mathcal W}_\pm\subset \tilde{W}_\pm$ such that $E^\pm|_{{\mathcal W}_\pm}:  {\mathcal W}_\pm\rightarrow  E^\pm({\mathcal W}_\pm)$ is a diffeomorphism. Moreover, we can assume that $E^\pm(\mathcal W_\pm)={\mathcal U'}_\pm\times U'_\pm$   with $U'_+\cap U'_-$ connected. 

 {\it Step 3: to prove the minimization of $\gamma_v$ to $\tilde P_v$ in an interval $[0,\varepsilon]$. } Let $\mathcal V$ be a convex neighborhood of $p_0$ with compact closure contained in $U'_+\cap U'_-$ which does not intersect the boundary of $W$, and choose $C>0$ such that $\tilde F(v)=F(-v)\leq C F(v)$ for every $v\in TM$ with $\pi(v)\in \mathcal V$.  Moreover, choose $\varepsilon>0$ small enough in such a way that $\bar B_{F}^-(p_0,(2+C)\varepsilon)\subset {\mathcal V}$  and a neighbourhood $\mathcal U\subset {\mathcal U'}_+\cap {\mathcal U'}_-$ of $v_0$ such that $\pi({\mathcal U})\subset  B_F^-(p_0,\varepsilon)\subset {\mathcal V}$. Then
 for every $v\in \mathcal U$ and  $t\in [0,\varepsilon]$, 
\begin{enumerate}
\item $\bar B_F^-(\gamma_{v}(t),\varepsilon)\subset{\mathcal V}$ (it follows from $\pi({\mathcal U})\subset  B_F^-(p_0,\varepsilon)$ and $\bar B_{F}^-(p_0,(2+C)\varepsilon)\subset \mathcal V$)  and as a consequence, $\bar B_F^-(\gamma_{v}(t),\varepsilon)$ is compact,
\item $\bar B_F^-(\gamma_{v}(t),\varepsilon)\cap (P_v\setminus U'_+)=\emptyset$ (it follows from $(1)$ and $\mathcal V\subset U'_+\cap U'_-$).
\end{enumerate}
 Moreover,  taking $\varepsilon$ smaller if necessary, we can also assume that
\begin{enumerate}
\item[(3)] all the orthogonal geodesics to $\tilde P_v$ of length at most $\varepsilon$ from the opposite side of $\tilde  P_v$ to $v$ do not intersect $\gamma_{v}:[0,\varepsilon]\rightarrow M$,
\end{enumerate}
this is because $U'_+\cap U'_-\setminus \tilde P_v$ has two connected components and the orthogonal geodesics from the opposite side are in a different component than $\gamma_v|_{[0,\varepsilon]}$ (recall that the restriction of $E^-(\cdot,\cdot,v)$ to a certain domain is a diffeomorphsim with  image $U'_+\cap U'_-$ for every $v\in\mathcal U$). 
Consider any $v\in \mathcal U$ and observe that by the continuity of $d_F$,  for every $t\in [0,\varepsilon]$, the distance $d_F(\tilde P_v,\gamma_{v}(t))$ is attained at some point $q\in \tilde P_v\cap \bar B_F^-(\gamma_{v}(t),\varepsilon)$. Indeed, if $p=\gamma_v(0)$, as $d_F(p,\gamma_{v}(t))\leq \varepsilon$ and  $\tilde P_v\cap \bar B_F^-(\gamma_{v}(t),\varepsilon)$ is compact (recall that $\bar B_F^-(\gamma_{v}(t),\varepsilon)\subset \mathcal V $ and $\mathcal V$ does not intersect the boundary of $\tilde P_v$, which is contained in the boundary of $W$), Weierstrass Theorem ensures the existence of $q$. Assume now that $p\not= q$. Then as $q,\gamma_{v}(t)\in \mathcal V$, there exists a minimizing geodesic $\tilde\gamma$ from $q$ to $\gamma_{v}(t)$ with length equal to $d_F(\tilde P_v,p)$. By \cite{CJM11}, $\tilde\gamma$ has to be orthogonal to $\tilde P_v$, but taking into account the condition $(3)$ above, this is a contradiction with the fact that the map $E$ is a diffeomorphism when the image is restricted to $v\times U'$. Finally,  condition $(2)$ above guarantees that $\gamma_v:[0,\varepsilon]\rightarrow M$ also minimizes the distance from $P$ to $\gamma_v(t)$, taking into account that $P\subset \tilde P_v$.
\end{proof}

\begin{corollary}\label{precompactQ} Let $P$ be an embedded submanifold of a Finsler manifold $(M,F)$.  If $Q$ is a pre-compact open set of $P$, then there exists $\varepsilon>0$ such that $i_P(v)>\varepsilon$ for all $v\in \nu^1(Q).$ 
\end{corollary}

\begin{proof} Assume that there is a sequence $v_n\in \nu^1(Q)$ such that $\lim_{n\to \infty}i_P(v_n)=0.$ By compactness of $\{v\in \nu^1(P); \pi(v)\in \overline{Q}\}$, we can take a convergent subsequence of $v_n$, say $\lim v_n=v_0\in \nu^1(P)$. This is a contradiction, because by the previous theorem, there exists   an open neighbourhood $\mathcal U$ of $v_0$ in $\nu^1(P)$ such that $i_{P}(v)>\varepsilon$ for some $\varepsilon>0$  and for all $v\in \mathcal U$. 
 \end{proof}
  Let us recall that roughly speaking, an orthogonal geodesic to $P$, $\gamma:[0,1]\rightarrow M$, has a $P$-focal instant $t_0\in(0,1]$ if it admits a variation by   geodesics that depart orthogonally from $P$   and fix $\gamma(t_0)$ up to first order, or if there exists a $P$-Jacobi field $J$ such that $J(t_0)=0$. For precise definitions see \cite[\S 3,4]{JavSoa15}. 
 \begin{proposition}\label{localdifeo}
Let $(M,F)$ be a Finsler manifold and $P$ an embedded submanifold of $M$. If ${\mathcal V}\subset \nu(P)$ is an open neighbourhood  where the exponential map of $F$ is well-defined and for all $v_0\in {\mathcal V}$ the geodesic $\gamma_{v_0}:[0,1]\rightarrow M$ does not have any $P$-focal point, then the exponential map $\exp: {\mathcal V}\subset \nu(P)\rightarrow M$  is a local diffeomorphism in $v_0$. 
 \end{proposition}
 \begin{proof}
 It follows from a straightforward computation. Observe that the exponential map has a singularity in $v\in \nu(P)$ if and only if there exists a smooth curve  $w:(-\epsilon,\epsilon)\rightarrow \nu(P)$ with $w(0)=v$ such that $\frac{\partial }{\partial t} \exp(w(t))|_{t=0}=0$, but in such a case, it is  not difficult to check that $J(s)=\frac{\partial}{\partial t}\gamma_{w(t)}(s)$ is a $P$-Jacobi field along $\gamma_v$ with $J(1)=0$ (see \cite[Proposition 3.15]{JavSoa15}) and then $\gamma_v:[0,1]\rightarrow M$ has a $P$-focal point.
 \end{proof}
 
 \begin{proposition}\label{focalpoints}
Let $P$ be a k-dimensional embedded submanifold in a Finsler manifold $(M^n,F)$. Assume that the exponential map of $F$ is defined in $v_0\in \nu^1(P)$. Then there exists a neighbourhood  ${\mathcal U}\subset\nu^1(P)$ of $v_0$ and $\varepsilon>0$ such that $\gamma_v:[0,\varepsilon]\rightarrow M$ is well-defined and does not have any $P$-focal point for every $v\in{\mathcal U}$.
\end{proposition}
\begin{proof}
This result can be proved analogously to the proof of the existence of a first focal point in \cite{AJ16} (before Lemma 5.4), which in particular proves that there is an interval $[0,\varepsilon]$ where $\gamma_{v_0}$ has no $P$-focal points. As there is a smooth dependence  on $v$ (the initial velocity of the geodesic $\gamma_v$),  one can ensure the existence of a neighbourhood ${\mathcal U}\subset\nu(P)$ such that $\gamma_v$ does not have $P$-focal points in $[0,\varepsilon]$ as required.
\end{proof}
 The next result ensures the existence of the so-called {\it geometric tubular neighbourhoods}, where the open subset of orthogonal vectors is determined by those with Finsler length less than some $\varepsilon>0$. 
\begin{theorem}\label{tubneigh}
Let $(M,F)$ be a Finsler manifold and $P$ an embedded submanifold of $M$. Then given a precompact open subset $Q$ of $P$, there exists $\varepsilon>0$ such that the exponential map is defined in the open subset
\[{\mathcal V}_\varepsilon=\{ v\in \nu(Q): F(v)<\varepsilon\},\]
$\exp:\mathcal V_\varepsilon\rightarrow V\setminus Q$ is a diffeomorphism for a certain open subset $V_\varepsilon\subset M$ and the geodesic $\gamma_v:[0,1]\rightarrow M$ minimizes the distance from $P$ to $\gamma_v(1)$ for all $v\in {\mathcal V}_\varepsilon$. In particular, ${\mathcal V}_\varepsilon$ is a tubular neighbourhood of $Q$.
\end{theorem}
\begin{proof}
Choose $\varepsilon>0$ given by Corollary \ref{precompactQ}. In particular, by Lemma \ref{lemma3} we obtain that $\exp:{\mathcal V}_\varepsilon\rightarrow M$ is injective. Moreover, Proposition \ref{focalpoints} allows us to take $\varepsilon>0$ smaller if necessary in such a way that there are no $P$-focal points in $\gamma_v:[0,1]\rightarrow M$ for all $v\in {\mathcal V}_\varepsilon$. Then by Proposition \ref{localdifeo}, $\exp|_{{\mathcal V}_\varepsilon}$ is an injective local diffeomorphism, and then a diffeomorphism.
\end{proof}
\begin{corollary}
Let $(M,F)$ be a Finsler manifold and $P$ an embedded submanifold of $M$. Then there exists an open subset of $P$ where the distance from $P$ is smooth.
\end{corollary}
\begin{remark}
A similar result can be obtained for the distance to $P$ by considering the reverse Finsler metric $\tilde F$ defined as $\tilde F(v)=F(-v)$ for every $v\in TM$. Taking into account the reverse metric $\tilde F$ we  can also obtain a sort of forward tubular neighbourhood  in which we will use the backward exponential. 
\end{remark}



\section*{Acknowledgements} The authors warmly acknowledge discussions with  Professor Marcos Alexandrino about the results on this paper.

\end{document}